\documentclass[11pt]{article}%
\usepackage{amsmath,amsfonts,amsthm,amssymb, graphicx}
\usepackage{subfig}
\usepackage{amsmath}
\usepackage{amsfonts}
\usepackage{amssymb}
\usepackage{graphicx}%
\providecommand{\U}[1]{\protect\rule{.1in}{.1in}}

\pdfoutput=1
\theoremstyle{plain}

\newtheorem{proposition}{Proposition}
\newtheorem{lemma}{Lemma}

\begin{document}

\title{Sampling Point Processes on Stable Unbounded Regions and Exact Simulation of Queues}

\author{Jose Blanchet, Jing Dong \\ [12pt]
        Industrial Engineering and Operations Research Department \\
        Columbia University\\
        New York, NY 10027, USA}
        
\date{}
        
\maketitle

\section*{Abstract}
Given a marked renewal point process (assuming that the marks are i.i.d.) we
say that an unbounded region is stable if it contains finitely many points of
the point process with probability one. In this paper we provide algorithms
that allow to sample these finitely many points efficiently. We explain how
exact simulation of the steady-state measure valued state descriptor of the
infinite server queue follows as a simple corollary of our algorithms.
We provide numerical evidence supporting that
our algorithms are not only theoretically sound but also practical. Finally, having simulation optimization in mind,
we also apply our results to gradient estimation of steady-state performance measures.

\section{Introduction} \label{sec:intro}

Let $N=\{  N\left(  t\right)  :t\in\left(  -\infty,\infty\right)  \}
$ be a two sided time stationary renewal point process. We write $\{A_{n}%
:n\in\mathbb{Z}_{0}\}$ for the times at which the process $N$ jumps, where
$\mathbb{Z}_{0}=\mathbb{Z}\backslash\{0\}$ denotes the set of integers
removing zero, and with $A_{1}>0>A_{-1}$. For simplicity we assume that
$A_{n}<A_{n+1}$ for every $n$. Further, we define $X_{n}=A_{n+1}-A_{n}$.

Now let $\{V_{n}:n\in\mathbb{Z}_{0}\}$ be a sequence of independent and
identically distributed (i.i.d.) random variables (r.v.'s) which are
independent of the process $N$. Define $Z_{n}=\left(  A_{n},V_{n}\right)  $
and consider the marked point process $\mathcal{M}=\{Z_{n}:n\in\mathbb{Z}%
_{0}\}$ which forms a subset of $\mathbb{R}^{2}$. We say that a (Borel
measurable) set $\mathcal{B}$ is stable if $\left\vert \mathcal{M}%
\cap\mathcal{B}\right\vert <\infty$ almost surely (where $\left\vert
\mathcal{C}\right\vert $ is used to denote the cardinality of the set
$\mathcal{C}$).

Under natural assumptions on the inter-arrival times underlying $N$ and on the
distribution of the $V_{n}$'s (stated in Section~\ref{sec:alg}) we propose and study a
class of algorithms that allow to sample exactly (i.e. without any bias) a
realization of the set $\mathcal{M}\cap\mathcal{B}$ for a large class of
unbounded, stable sets $\mathcal{B}$.

Our approach builds on algorithms that are fully developed and studied in
\cite{BlanDon:2012}. As an application of the class of algorithms that we
study here, we provide a procedure that allows to sample from the steady-state
measure valued descriptor of an infinite server queue without any bias (i.e.exact simulation).
Such a procedure, for instance, is obtained by considering the particular case in which
$\mathcal{B}$ takes the form $\mathcal{B}=\{\left(  t,v\right)  :v>\left\vert
t\right\vert ,t\leq0\}$. Given that point processes constitute a natural way
of constructing queueing models in great generality, we believe that the class
of algorithms that we propose here have the potential to be applicable to the
design of exact sampling algorithms of more general queueing models. This is a
research avenue that we plan to investigate in the future.

We argue empirically that it is cheaper to run our exact sampling procedure to fully delete the
initial bias than it is to do a burn-in period that reduces the bias to a
reasonable size, say $5\%$, when talking about, for instance, the steady-state
queue length.

Finally, we apply our exact sampling algorithms for infinite server queues to
perform steady-state sensitivity analysis. For instance, we consider
quantities such as the derivative of the steady-state average
remaining service time with respect to the arrival rate or service rate. These
quantities are of great interests in stochastic optimization via simulation.

So, in summary, our contributions are as follows:

i) We provide the first exact sampling algorithm for stationary marked renewal
processes on unbounded and stable sets, see Section~\ref{sec:alg}.

ii) As a corollary of i) we explain how to obtain an exact sampling algorithm
for the steady-state measure valued descriptor of the infinite server queue.
We also show empirically that this algorithm is
\textit{practical} in the sense of being both easy to code and fast to run, see Section~\ref{sec:inf}.

iii) Finally, we provide new procedures for the sensitivity analysis of
steady-state performance measures of the infinite server queue, see Section~\ref{sec:sens}.

\subsection*{Relevant literature}

Following the seminal work by \cite{PropWil:1996}, several exact sampling
algorithms have been developed, particularly for spatial point processes.
\cite{Ken:1998} and \cite{KenMol:2000} developed algorithms and analytical
tools based on so-called Dominated Coupling From the Past (DCFP). DCFP is
based on the idea of introducing a stationary dominating process that is
simulatable. Compared to our method, firstly they use spatial birth and death
processes (generally of poisson type) as the coupled dominating processes.
This would limit the target distribution to be absolutely continuous with
respect to the Poisson measure. Secondly the number of steps simulated in the
naive DCFP grows exponentially with the system scale (i.e. arrival rate in the infinite server queue setting);
see Proposition 1 in \cite{BertMol:2002} for a detailed proof.
Although several modifications have been proposed, still the
number of steps involved in these backward construction appears to be
significantly large, especially when sampling in infinite volume regions \cite{FerFer:2002};
see Section 7 in  \cite{BertMol:2002} for empirical comparisons.

Our method is based on a construction that is being used in \cite{BlanSig:2011}  and
\cite{BlanDon:2012}; see also \cite{EnsGlyn:2000} for related ideas. The method involves the technique of simulating the
maximum of a negative drift random walk and the last passage time of
independent and identically distributed random variables to an increasing boundary. As shown in \cite{BlanDon:2012} the
complexity of our algorithm scales graciously as the system scale grows.

\section{Sampling form stable unbounded regions} \label{sec:alg}

We start by discussing the assumptions behind our development.\\

\noindent{\bf Assumptions:}

{\bf A1)} Assume that $E\left\vert V_{n}\right\vert ^{1/\alpha}<\infty$ for some
$\alpha> 0$, we also write $F\left(  \cdot\right)  =P\left(  V_{n}\leq
\cdot\right)  $ for the cumulative distribution function (CDF) of $V_{n}$ and
put $\overline{F}\left(  \cdot\right)  =1-F\left(  \cdot\right)  $ for the
tail CDF.

{\bf A2)} We assume that $F\left(  \cdot\right)  $ is known and easily accessible
either in closed form or via efficient numerical procedures. Moreover, we can
simulate $V_{n}$ conditional on $V_{n}\in [a,b]$ with $P\left(  V_{n}
\in [a,b] \right)  >0$. Finally we can find $u(k)$ such that $u(k) \geq \int_{k}^{\infty} P(|V_1|^{1/\alpha}>\nu)d\nu$ and $u(k) \rightarrow 0$ as $k\rightarrow \infty$.

{\bf A3)} Recall that $X_{n}=A_{n+1}-A_{n}>0$. Define $\psi\left(  \theta\right)
=\log E\exp\left(  \theta X_{n}\right)  $ and assume that there exists
$\delta>0$ such that $\psi\left(  \delta\right)  <\infty$. Finally, let us
write $\mu=EX_{n}$.

{\bf A4)} Define $G\left(  \cdot\right)  =P\left(  X_{n}\leq\cdot\right)  $ and
$\overline{G}\left(  \cdot\right)  =1-G\left(  \cdot\right)  $. Suppose that
$G\left(  \cdot\right)  $ is known and that it is possible to simulate from
$G_{eq}\left(  \cdot\right)  :=\mu^{-1}\int_{\cdot}^{\infty}\overline
{G}\left(  t\right)  dt$. Moreover, let $G_{\theta}\left(  \cdot\right)
=E\exp\left(  \theta X_{n}-\psi\left(  \theta\right)  \right)  I\left(
X_{n}\leq\cdot\right)  $ be the associated exponentially tilted distribution
with parameter $\theta$ for $\psi\left(  \theta\right)  <\infty$. We assume
that we can simulate from $G_{\theta}\left(  \cdot\right)  $.

\bigskip

Consider the class of sets $\mathcal{B}\subset \mathbb{R}^{2}$ that are Borel
measurable and such that
\[
\mathcal{B}\subset\mathcal{C}_{\alpha}=\{(t,v):\left\vert v\right\vert
\geq|t|^{\alpha}\}.
\]

Our goal in this section is to develop an algorithm that allows to sample
without bias the random set $\mathcal{M}\cap\mathcal{C}_{\alpha}$, and
therefore $\mathcal{M}\cap\mathcal{B}$. We will discuss extensions that follow
immediately from our formulation at the end of this section. Figure~\ref{fig: rua} illustrates the different shapes that the set $\mathcal{C}_{\alpha}$
can take depending on the values of $\alpha>0$.

\begin{figure}[htb]
{
\centering
\includegraphics[width=0.50\textwidth]{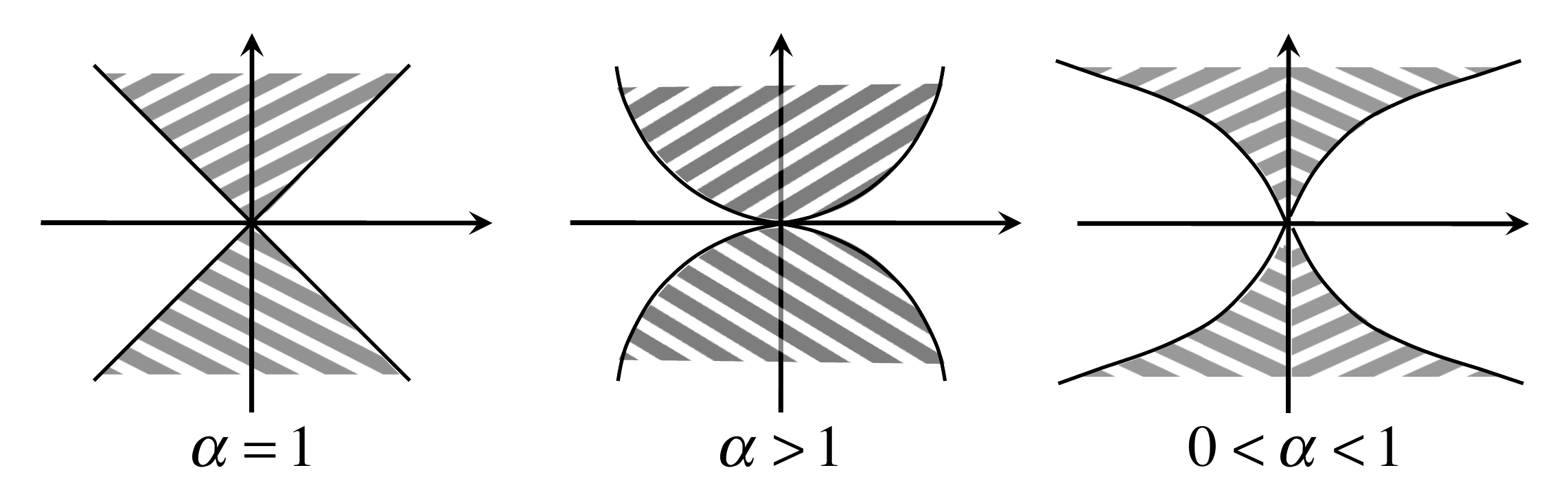}
\caption{The area of $\mathcal{C}_{\alpha}$. The horizontal axis corresponds to the $t$ coordinate while the vertical axis represents the $v$ coordinate \label{fig: rua}}
}
\end{figure}

We now proceed to explain our construction. As the stationary renewal point process is time reversible, starting at $0$
the distribution of the forward process $\{Z_{n}: n > 0\}$ and the backward
process $\{Z_{n}: n<0\}$ are the same. In what follows we limit our discussion
to the construction of the forward process and the simulation of the backward
process is completely analogous.

We follow the idea in \cite{BlanDon:2012}. Let $\epsilon\in(0, \mu)$.
Consider any random time $\kappa$, finite with probability one but large
enough such that

\begin{center}
$A_{n+1}\geq n(\mu-\epsilon)$ and $\left\vert V_{n+1}\right\vert \leq
(n(\mu-\epsilon))^{\alpha}$
\end{center}
for all $n \geq\kappa$.

If such random time $\kappa$ is well defined, we
only need to simulate the stationary process up to $\kappa$ to get a sample
from the unbounded region.\\

\begin{proposition}
The random time $\kappa$ defined above exists and it is finite with
probability one.
\end{proposition}

\begin{proof}
By Chebyshev's inequality,
$$P(A_{n+1} < n(\mu-\epsilon)) \leq E[\exp(\theta(n(\mu-\epsilon)-A_{n+1}))) \leq \exp(-n(-\theta(\mu-\epsilon)-\psi(-\theta)))$$
for any $\theta \geq 0$.\\
Let
$$I(-\epsilon)=\max_{\theta \geq 0}\{-\theta(\mu-\epsilon)-\psi(-\theta)\}$$
As $\psi(0)=0$, $\psi'(0)=\mu$ and $\psi''(0)=Var(X)>0$, $I(-\epsilon)>0$. Then
$$P(A_{n+1}<n(\mu-\epsilon))\leq\exp(-nI(-\epsilon))$$
and
$$\sum_{n=1}^{\infty} P(A_{n+1} < n(\mu-\epsilon)) \leq \frac{\exp(-I(-\epsilon))}{1-\exp(-I(-\epsilon))} < \infty$$
By Borel-Cantelli lemma, $\{A_{n+1} \geq n(\mu-\epsilon)\}$ eventually almost surely.\\
Similarly and independently we have
\begin{eqnarray*}
\sum_{n=1}^{\infty} P(\left\vert V_{n+1} \right\vert>(n(\mu-\epsilon))^{\alpha}) = \sum_{n=1}^{\infty} P(\left\vert V_1 \right\vert^{1/\alpha}>n(\mu-\epsilon))
\leq \frac{1}{\mu-\epsilon}\int_{0}^{\infty} P(\left\vert V_1 \right\vert^{1/\alpha}>\nu)d\nu < \infty
\end{eqnarray*}
Thus, again by Borel-Cantelli lemma, $\{\left\vert V_{n+1} \right\vert \leq (n(\mu-\epsilon))^{\alpha}\}$ eventually almost surely. Therefore,
$P(\kappa<\infty)=1$
\end{proof}

As $\{A_{n}:n\geq1\}$ and $\{V_{n}:n\geq1\}$ are
independent of each other, we consider the following construction. Let
$\kappa(A)$ be a random time satisfying that $A_{n+1}\geq n(\mu-\epsilon)$ for
$n\geq\kappa(A)$, and $\kappa(V)$ be a random time satisfying that
$V_{n+1}\leq n(\mu-\epsilon)$ for $n\geq\kappa(V)$. Clearly $\kappa\left(
A\right)  $ and $\kappa\left(  V\right)  $ are \textit{not }stopping times and
this makes the simulation of these times challenging. However, we will explain
how to sample these times and then we can set $\kappa=\max\{\kappa
(A),\kappa(V)\}$. Our construction will allow us to simulate
$\{A_{n}:n\geq1\}$ and $\{V_{n}:n\geq1\}$ separately.

\subsection{Simulation of $\{ A_{k}: 1\leq k\leq\max\{n,\kappa\left(A\right)\}+1\}  $}

In this subsection we will introduce a method to simulate $\kappa(A)$ together
with $\{A_{k}: k \geq 1\}$.

First, define $A_{1}$ according to the distribution $G_{eq}\left(
\cdot\right)  $. Sampling $A_1$ can be done according to A4).

Now, observe that $A_{n+1}=A_{1}+X_{1}+...+X_{n}$ and define%
\[
\tilde{S}_{n}=n(\mu-\epsilon)-\left(  A_{n+1}-A_{1}\right)  =\sum_{i=1}%
^{n}Y_{i},
\]
where $Y_{i}=(\mu-\epsilon)-X_{i}$. Note that the $Y_{i}$'s are i.i.d. with
$EY_{i}=-\epsilon$. If we set $\tilde{S}_{0}=0$, then $\{\tilde{S}_{n}:n\geq0\}$
is a random walk with negative drift. We are interested in sampling up to the
\textit{last time} $n$ at which $\tilde{S}_{n} > 0$.

We define the following sequence of random times:
\[
\Delta_{1}=0,\text{ }\Gamma_{1}=\inf\{n\geq\Delta_{1}:\tilde{S}_{n}-\tilde
{S}_{\Delta_{1}}>0\},
\]
and for $j\geq2$
\begin{align*}
\Delta_{j} &=\inf\{n\geq\Gamma_{j-1}\mathbf{1}\{\Gamma_{j-1}<\infty
\}\vee\Delta_{j-1}:\tilde{S}_{n}\leq0\},\\
\Gamma_{j} &=\inf\{n\geq\Delta_{j}:\tilde{S}_{n}-\tilde{S}_{\Delta_{j}%
}>0\}.
\end{align*}

Now, let $\gamma=\inf\{j\geq1:\Gamma_{j}=\infty\}$ and note that
$\Delta_{\gamma+1}=\Delta_{\gamma}$ and that $\tilde{S}_{n}\leq0$ for
$n\geq\Delta_{\gamma}$, which in particular implies that $A_{n+1}\geq
n(\mu-\epsilon)$ for $n\geq\Delta_{\gamma}$. Therefore, we have that
$\Delta_{\gamma}=\kappa\left(  A\right)  $.

In what follows we will explain how to simulate the $\Delta_{j}$'s and
$\Gamma_{j}$'s sequentially and jointly with the underlying random walk until
time $\Delta_{\gamma}$. One important observation is that for every $j \geq 1$, $\Delta_{j}<\infty$
almost surely by the strong law of large numbers.

Let us write $\mathcal{F}_{n}=\sigma\{Y_{1},Y_{2},...,Y_{n}\}$ for the
$\sigma$-field generated by the $Y_{j}$'s up to time $n$. Let $\xi\geq0$ and
define
\[
T_{\xi}:=\inf\{n\geq0:\tilde{S}_{n}>\xi\},
\]
then by the strong Markov property we have that for $j\leq\gamma$,
\[
P(\Gamma_{j}=\infty|\mathcal{F}_{\Delta_{j}})=P(\Gamma_{j}=\infty|\tilde
{S}_{\Delta_{j}})=P(T_{0}=\infty)>0,
\]
where we use $P\left(  \cdot\right)  $ to denote the nominal probability measure under
which $\tilde{S}_{0}=0$.

It is important then to note that
\[
P\left(  \gamma=k\right)  =P\left(  T_{0}<\infty\right)  ^{k-1}P\left(
T_{0}=\infty\right)
\]
for $k\geq1$. In other words, $\gamma$ is geometrically distributed. The
procedure that we have in mind is to simulate $\Delta_{\gamma}$ in time
intervals, and the number of time intervals is precisely $\gamma$.

Let $\psi_{Y}(\theta)=\log E\exp(\theta Y_{i})$. As the moment generating
function of $X_{i}$ is finite in a neighborhood of the zero, $\psi_{Y}(\cdot)$
is also finite in a neighborhood of zero and $EY_{i}=\psi_{Y}^{\prime
}(0)=-\epsilon$, Var$(Y_{i})=\psi_{Y}^{\prime\prime}(0)>0$. Then by the
convexity of $\psi_{Y}(\cdot)$, one can always select $\epsilon>0$
sufficiently small so that there exists $\eta>0$ with $\psi_{Y}(\eta)=0$ and
$\psi_{Y}^{\prime}(\eta)>0$. The root $\eta$ allows us to define a new measure
$P_{\eta}$ based on exponential tilting so that%
\[
\frac{dP_{\eta}}{dP}\left(  Y_{i}\right)  =\exp(\eta Y_{i}).
\]
Moreover, under $P_{\eta}$, $\tilde{S}_{n}$ is random walk with positive drift
equal to $\psi_{Y}^{\prime}\left(  \eta\right)$ (\cite{Asm:2003} P. 365). Therefore $P_{\eta}%
(T_{0}<\infty)=1$ and $P(T_{0}<\infty)=E_{\eta}(\exp(-\eta\tilde{S}_{T_{0}}%
))$. More generally, $P_{\eta}(T_{\xi}<\infty)=1$ and
\[
q\left(  \xi\right)  :=P(T_{\xi}<\infty)=E_{\eta}(\exp(-\eta\tilde{S}_{T_{\xi
}}))
\]
for each $\xi\geq0$. Based on the above analysis we now introduce a convenient
representation to simulate a Bernoulli random variable $J\left(  \xi\right)  $
with parameter $q\left(  \xi\right)  $ namely,
\begin{equation}
J\left(  \xi\right)  =I(U\leq\exp(-\eta\tilde{S}_{T_{\xi}})). \label{ConRep}%
\end{equation}
where $U$ is a uniform random variable independent of everything else under
$P_{\eta}$.

Identity (\ref{ConRep}) provides the basis for an implementable algorithm to
simulate a Bernoulli with success probability $q(\xi)$. Sampling $\{\tilde
{S}_{1},...,\tilde{S}_{T_{0}}\}$ conditional on $T_{0}<\infty$, as we shall
explain now, corresponds to basically the same procedure. First, let us write $P^{\ast}(\cdot
)=P(\cdot|T_{0}<\infty)$. The following result provides an expression for the
likelihood ratio between $P^{\ast}$ and $P_{\eta}$.

\begin{lemma}
We have that
\[
\frac{dP^{\ast}}{dP_{\eta}}(\tilde{S}_{1},...,\tilde{S}_{T_{0}})=\frac
{\exp(-\eta\tilde{S}_{T_{0}})}{P(T_{0}<\infty)}\leq\frac{1}{P(T_{0}<\infty)}.%
\]

\end{lemma}

\begin{proof}
\begin{eqnarray*}
P(\tilde S_1 \in H_1, ... , \tilde S_{T_0} \in H_{T_0} | T_0 < \infty) &=& \frac{P(\tilde S_1 \in H_1, ... , \tilde S_{T_0} \in H_{T_0}, T_0 < \infty)}{P(T_0 < \infty)}\\
&=&\frac{E_{\eta}[\exp(-\eta \tilde S_{T_0})I(\tilde S_0 \in H_0, ... , \tilde S_{T_0} \in H_{T_0})]}{P(T_0 < \infty)}.
\end{eqnarray*}
\end{proof}

The previous lemma provides the basis for a simple acceptance / rejection
procedure to simulate $\{\tilde{S}_{1},...,\tilde{S}_{T_{0}}\}$ conditional on
$T_{0}<\infty$. More precisely, we propose $(\tilde{S}_{1},...,\tilde
{S}_{T_{0}})$ from $P_{\eta}\left(  \cdot\right)  $. Then one generates a
uniform random variable $U$ independent of everything else and accept the
proposal if
$$U\leq\frac{1}{1/P(T_{0}<\infty)}\times\frac{dP^{\ast}}{dP_{\eta}}(\tilde
{S}_{1},...,\tilde{S}_{T_{0}})=\exp(-\eta\tilde{S}_{T_{0}}).$$
This criterion coincides with $J\left(  0\right)  $ according to
(\ref{ConRep}). So, the procedure above
simultaneously obtains both a Bernoulli r.v. $J\left(  0\right)  $ with
parameter $q\left(  0\right)  $, and the corresponding path $\{\tilde{S}%
_{1},...,\tilde{S}_{T_{0}}\}$ conditional on $T_{0}<\infty$.\newline

\noindent\textbf{Algorithm 1 (Outputs $(\tilde{S}_{0},...,\tilde{S}_{\Delta_{\gamma}})$)}

\begin{itemize}
\item[Step 0.] Set $K=0$, and $S_{0}=0$

\item[Step 1.] Simulate $(\tilde{S}_{1},...,\tilde{S}_{T_{0}})$ from $P_{\eta
}$ and compute $J:=J\left(  0\right)  $ according to (\ref{ConRep}).

\item[Step 2.] If $J=1$, then let $S_{K+j}=\tilde{S}_{j}$ for $j=1,...,T_{0}$
and update $K\longleftarrow K+T_{0}$. Then, go back to Step 1.\newline Otherwise,
$J=0$ (i.e. $\Delta_{\gamma}=K$), stop and output $\left(  S_{0},...,S_{K}\right)$
\end{itemize}

\noindent{\bf Remark:} It has been proved in  \cite{BlanDon:2012} that the expected number of times we need to repeat Step 1 does not change with the system scale (i.e. the arrival rate).

We noted earlier that $\Delta_{\gamma}=\kappa\left(  A\right)  $ and Algorithm
1 together with the initial procedure to sample $A_{1}$ allows us to simulate
$\left(  A_{j+1}:0\leq j\leq\kappa\left(  A\right)  \right)  $, and we know
that $A_{n+1}\geq n(\mu-\epsilon)$ for $n\geq\kappa\left(  A\right)  $. We
need to simulate $A_{n+1}$ for $n\leq\kappa=\max\{\kappa\left(  A\right)
,\kappa\left(  V\right)  \}$, and $\kappa\left(  V\right)  $ is independent of
$\kappa\left(  A\right)  $. So, there might be cases for which we will have to
sample $A_{n+1}$ for $n>\kappa\left(  A\right)  $. Since\ $A_{n+1}%
=A_1-\tilde S_n+n(\mu-\epsilon)$ it suffices to explain how to simulate
$\tilde{S}_{n}$ for $n>\Delta_{\gamma} $. In turn, it suffices to
explain how to simulate $(\tilde{S}_{n}:n\geq0)$ with $\tilde{S}_{0}=0$
conditional on $T_{0}=\infty$. We will once again apply an acceptance / rejection procedure
but this time we will use the original (nominal) distribution as the proposal
distribution. Define%
\[
P^{\prime}\left(  \cdot\right)  =P(\cdot|T_{0}=\infty).
\]
The following result provides an expression for the likelihood ratio between
$P^{\prime}$ and $P$.

\begin{lemma}
We have that
\[
\frac{dP^{\prime}}{dP}(\tilde{S}_{1},...,\tilde{S}_{l})=\frac{I(T_{0}%
>l)(1-q(-\tilde{S}_{l}))}{P(T_{0}=\infty)}\leq\frac{1}{P(T_{0}=\infty)}.%
\]

\end{lemma}

\begin{proof}
\begin{eqnarray*}
P(\tilde S_1 \in H_1, .... , \tilde S_l \in H_l | T_0 = \infty) &=& \frac{P(\tilde S_1 \in H_1, ... \tilde S_l \in H_l, T_0 = \infty)}{P(T_0 = \infty)}\\
&=& \frac{E[I(\tilde S_1 \in H_1, ... , \tilde S_l \in H_l) I(T_0 > l) P(T_0=\infty | \tilde S_0,...,\tilde S_l) ]}{P(T_0=\infty)}.
\end{eqnarray*}
The result then follows from the strong Markov property and homogeneity of the random walk.
\end{proof}

We are in good shape now to apply acceptance / rejection to sample from
$P^{\prime}$. The previous lemma indicates that to sample $\{\tilde{S}%
_{0},...,\tilde{S}_{l}\}$ given $T_{0}=\infty$ we can propose from the original
(nominal)\ distribution and accept with probability $q(-\tilde{S}_{l})$ as
long as $\tilde{S}_{j}\leq0$ for all $0\leq j\leq l$. So, in order to perform
the acceptance test we need to sample a Bernoulli with parameter $q(-\tilde
{S}_{l})$, but this is easily done using identity (\ref{ConRep}). Thus we
obtain the following procedure.\newline

\noindent\textbf{Algorithm 2 (Given $n\geq0$ outputs $\{A_{1},A_{2},...,A_{\max\{n,\kappa( A)\}+1}\}$)}

\begin{itemize}
\item[Step 1.] Run Algorithm 1 and obtain $\{S_{0},S_{1},...,S_{K}\}$.

\item[Step 2.] If $K=\kappa\left(  A\right)  \geq n$, jump to Step 6. Otherwise, $K<n$, let $l=n-K\geq1$.

\item[Step 3.] Simulate $\{\tilde{S}_{0},\tilde{S}_{1},...,\tilde{S}_{l}\}$ from
the original (nominal) distribution with $\tilde S_0=0$.

\item[Step 4.] If $\tilde{S}_{j}\leq0$ for all $0\leq j\leq l$ then sample a
Bernoulli $J(-\tilde{S}_{l})$ with parameter $q(-\tilde{S}_{l})$ using
(\ref{ConRep}) and continue to Step 5. Otherwise (i.e. $\tilde{S}_{j}>0$ for
some $1\leq j\leq l$) go back to Step 3.

\item[Step 5.] If $J(-\tilde{S}_{l})=1$, go back to Step 3. Otherwise,
$J(-\tilde{S}_{l})=0$, let $S_{K+i}=S_K+\tilde{S}_{i}$ for $i=1,2,...,l$

\item[Step 6.] Let $m=\max\{n,\kappa(A)\}$. Simulate $A_{1}$ with CDF
$G_{eq}(\cdot)=\mu^{-1}\int_{\cdot}^{\infty}\bar{G}(t)dt$. Set $A_{n+1}=A_1-S_n + n(\mu-\epsilon)$ for $n=1,...,m$.
Output $\{A_1,..., A_{m+1}\}$.

\end{itemize}

\subsection{Simulation of $\{ V_{n}: 1 \leq n \leq\kappa(V)+1\}$}

In this section we will introduce a method to simulate $\kappa(V)$ together
with the $\{V_{n}: n \geq 1\}$.

Let $p(n)=P(|V_{1}|>(n(\mu-\epsilon))^{\alpha})$. We define $\Upsilon_{0}=0$
and $\Upsilon_{i}=\inf\{n>\Upsilon_{i-1}:|V_{n+1}|>(n(\mu-\epsilon))^{\alpha
}\}$ for $i=1,2,...$. We also define two independent sequences of random
variables, $\{\hat{V}_{n+1}:n\geq1\}$, and $\{\bar{V}_{n+1}:n\geq1\}$ as follows.
The elements in each sequence are i.i.d., $\hat{V}_{n+1}$ is distributed as
$V_{n+1}$ conditional on $|V_{n+1}|>(n(\mu-\epsilon))^{\alpha}$, and $\bar{V}_{n+1}$
follows the distribution of $V_{n+1}$ conditional on $|V_{n+1}|\leq(n(\mu
-\epsilon))^{\alpha}$. We simulate $V_{1}$ following its nominal distribution
independent of everything else.

Let $\sigma=\inf\{i\geq0:\Upsilon_{i}=\infty\}$. Then $V_{n+1}\leq
(n(\mu-\epsilon))^{\alpha}$ for $n\geq\Upsilon_{\sigma-1}+1$. We next introduce a method to sample
$\Upsilon_{1},\Upsilon_{2},...$ sequentially and jointly with the $V_{n}$'s up
until $\Upsilon_{\sigma-1}$.

The following lemma provides the basis to guarantee the termination of our procedure.

\begin{lemma}
If $E|V_{1}|^{1/\alpha}<\infty$, then
\[
P(\Upsilon_{1}=\infty)=\prod_{i=1}^{\infty}(1-p(i))\geq\exp(-
2E|V_{1}|^{1/\alpha}/(\mu-\epsilon))>0,
\]
consequently $E\sigma\leq\exp(2E|V|^{1/\alpha}/(\mu-\epsilon))<\infty$.
\end{lemma}
\noindent{\bf Remark:} The bound on $E\sigma$ can be improved. This improvement is important for the theoretical asymptotic analysis of GI/GI/$\infty$ application, see \cite{BlanDon:2012}.
\begin{proof}
\begin{eqnarray*}
P(\Upsilon_1=\infty) = \prod_{n=1}^{\infty}(1-p(n)) &\geq& \prod_{n=1}^{\infty} \exp(-2p(n))\\
&\geq& \exp(-\frac{2}{\mu-\epsilon}\int_{0}^{\infty} P(|V_1|^{1/\alpha}>\nu)d\nu) = \exp(-\frac{2E|V_1|^{1/\alpha}}{\mu-\epsilon})
\end{eqnarray*}
For $i=2,3,...$
conditional on $\Upsilon(i-1)=k$:\newline%
\begin{equation*}
P(\Upsilon_{i}=\infty|\Upsilon_{i-1}=k)=\prod_{n=k+1}^{\infty}(1-p(n))\geq
\exp(-\frac{2\int_{k}^{\infty} P(|V_1|^{1/\alpha}>\nu)d\nu}{\mu-\epsilon}\geq
\exp(-\frac{2E|V_1|^{1/\alpha}}{\mu-\epsilon})%
\end{equation*}
Thus $\sigma$ is stochastically dominated by a geometric random variable with parameter $p=\exp(-2E|V_1|^{1/\alpha}/(\mu-\epsilon))$, the result then follows.

\end{proof}

Notice that
\begin{equation}
\prod_{i=k+1}^{l}(1-p(i))\geq P(\Upsilon_{i}=\infty | \Upsilon_{i-1}=k)\geq\prod_{i=k+1}%
^{l}(1-p(i))\times\exp(-\frac{2\int_{l}^{\infty} P(|V_1|^{1/\alpha}>\nu)d\nu}{\mu-\epsilon})\label{EG}
\end{equation}
for $l \geq k+1$.\\
Thus if we are simulating $I\sim$ Bernoulli$(r_{i})$ with $r_{i}%
:=P(\Upsilon_{i}=\infty|\Upsilon_{i-1})$, then with probability one we can check whether
$U\leq P(\Upsilon_{i}=\infty|\Upsilon_{i-1})$ for $U\sim$ Unif$[0,1]$ by making $l$
sufficiently large without calculating the infinite product in the definition
of $P(\Upsilon_{i}=\infty|\Upsilon_{i-1})$.

On the other hand, if we define $\prod_{j=1}^{0}(1-p(j)):=1$, then
\[
P(\Upsilon_{1}=n|\Upsilon_{1}<\infty)=p(n)\frac{\prod_{j=1}^{n-1}%
(1-p(j))}{P(\Upsilon_{1}<\infty)}\leq p(n)\frac{1}{P(\Upsilon_{1}<\infty)}.%
\]
Consider a random variable $N$ with the following probability density
function
\[
P(N=n)=cp(n)
\]
for $n=1,2,...$, where $c=(\sum_{n=1}^{\infty}p(n))^{-1}$.
Then
$
P(\Upsilon_{1}=n|\Upsilon_{1}<\infty)/P(N=n) \leq 1/(cP(\Upsilon
_{1}<\infty)).%
$

So we can simulate $\Upsilon_1$ given $\Upsilon_1<\infty$ using acceptance / rejection with $N$ as the proposal random variable. Generalizing the idea to $\Upsilon_{i}$, we can obtain the following algorithm \newline

\noindent\textbf{Algorithm 3 (Given $\Upsilon_{i-1}=k$, outputs $\Upsilon_{i}$ conditional on $\Upsilon_{i}<\infty$)}
\begin{itemize}
\item[Step 1.] Let $c=(\sum_{n=k+1}^{\infty} p(n))^{-1}$. Simulate $N$ with probability density function $P(N=n)=cp(n)$
for $n=k+1,k+2,...$

\item[Step 2.] Simulate $U\sim$ Unif$[0,1]$ independently. If $U\leq
\prod_{j=k+1}^{N-1}(1-p(j))$ , set $\Upsilon_{i}=N$ and stop. Otherwise go back to Step 1
\end{itemize}

We conclude this section with our procedure to simulate $\{V_{1},V_{2},...V_{\kappa(V)+1}\}$.\\

\noindent\textbf{Algorithm 4 (Outputs $\{V_{1},V_{2},...V_{\kappa(V)+1}\}$)}

\begin{itemize}
\item[Step 0.] Set $\Upsilon_{0}= 0$, $i=1$. Simulate $V_1$ from its nominal distribution.

\item[Step 1.] Simulate $I\sim$ Bernoulli$(r_{i})$ with $r_{i}:=P(\Upsilon
_{i}=\infty|\Upsilon_{i-1})$ (see (\ref{EG})).

\item[Step 2.] If $I=1$, set $\kappa(V)=\Upsilon_{i-1}+1$. Simulate
$V_{\kappa(V)+1}$ by sampling from $\bar{V}_{\kappa(V)+1}$ and stop. Otherwise
$I=0$, sample $\Upsilon_{i}$ conditional on $\Upsilon_{i}<\infty$ and the value of $\Upsilon_{i-1}$ using
Algorithm 3. Simulate the process between $\Upsilon_{i-1}+2$ and $\Upsilon_{i}+1$
by sampling from $\bar{V}_{n}$ for $\Upsilon_{i-1}+2 \leq  n \leq \Upsilon_{i}$ and
$\hat{V}_{n}$ for $n=\Upsilon_{i}+1$. Set $i=i+1$ and then go back to Step 1.
\end{itemize}

\section{Application to the infinite server queue} \label{sec:inf}

As a direct application of the ideas discussed in the previous section we
study steady-state simulation for the infinite server queue. The following
diagram indicates how to construct the steady-state measure valued descriptor
assuming that we can sample all the points inside the set
\[
\mathcal{C}=\{(t,v):v\geq\left\vert t\right\vert ,t\leq0\}.
\]
Let $Q(t,y)$ denote the number of people in the system at time $t$ with
residual service time strictly greater than $y$ and $E(t)$ denote the time
elapsed since the previous arrival at time $t$ (i.e. $E\left(  \cdot\right)  $
is the age process associated with $N\left(  \cdot\right)  $). Figure~\ref{fig: rub} below depicts the region $\mathcal{C}$. Every point in $\left\vert
\mathcal{M}\cap\mathcal{C}\right\vert $ is projected to the vertical line at
time zero by drawing a $-$45$^{0}$ line. The final position in the
vertical line if positive, represents the corresponding remaining service time. Since the
underlying point process is time stationary, the whole configuration of points
obtained by this procedure at time zero is a snap shot of the steady-state
distribution of the infinite server queue.

\begin{figure}[htb]
{
\centering
\includegraphics[width=0.3\textwidth]{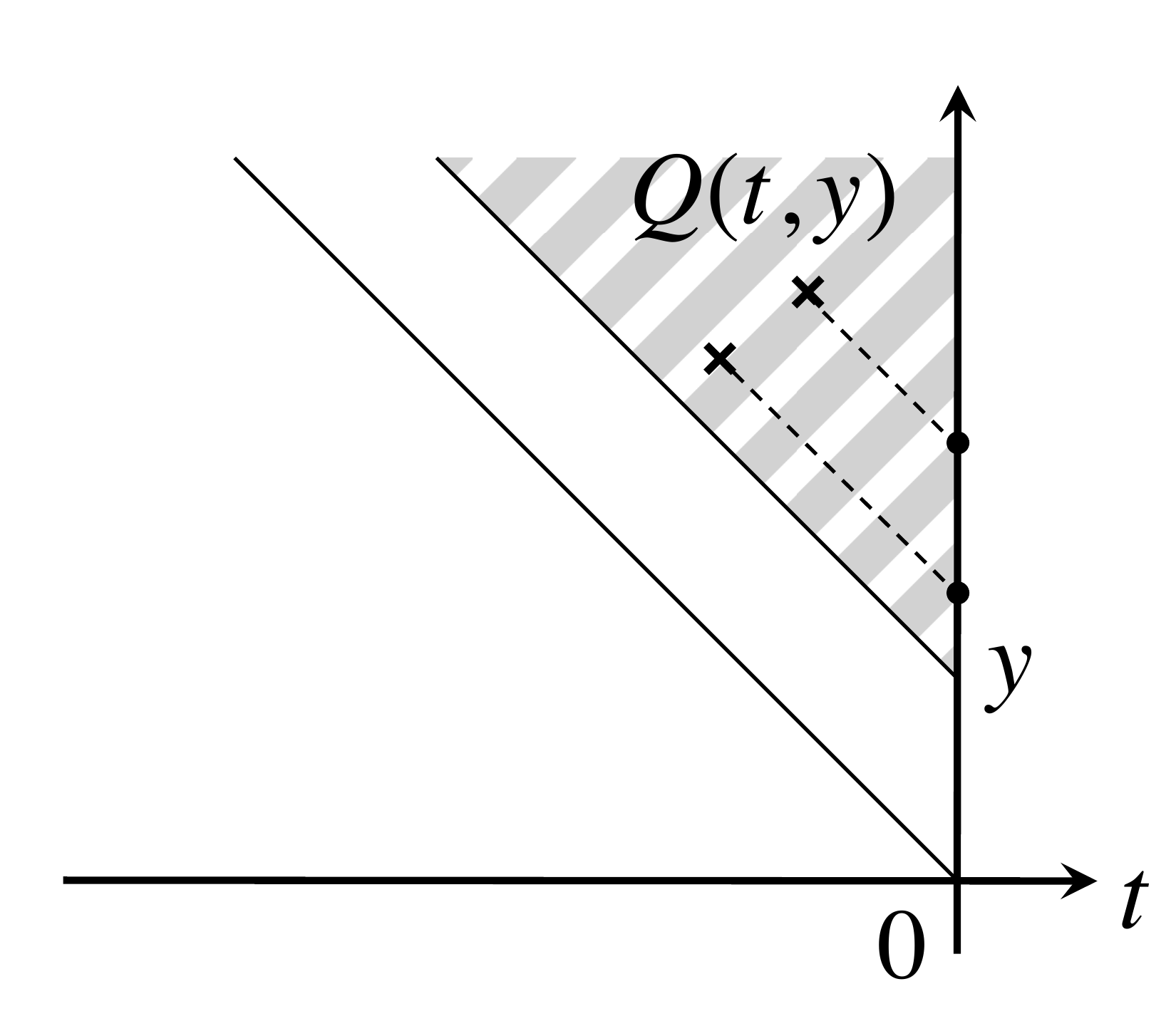}
\caption{The points lies in the shaded area correspond to people who are still in the system at time $0$ with remaining service time greater than $y$}
\label{fig: rub}}
\end{figure}

\subsection{Algorithm for the Infinite Server Queue}
As depicted in Figure~\ref{fig: rub} after projecting into the vertical line at $t=0$, we obtain the stationary remaining service requirements of the customers at time zero. We shall use $R_1, R_2, ..., R_{Q(0,0)}$ to denote the remaining service times. The labeling is arbitrary although we will assign smaller indexes to customers that have spent less time in the system. Our algorithm proceeds as follows.\\

\noindent\textbf{Algorithm 5 (Outputs $\{R_1, R_2, ... ,R_{Q(0,0)}\}$ and $E(0)$)}

\begin{itemize}

\item[Step 1.] Use Algorithm 4 to simulate the $\{V_{n}, 1 \leq n \leq
\kappa(V)+1\}$.

\item[Step 2.] Use Algorithm 2 to simulate the $\{A_{1},A_{2},...,A_{\max\{\kappa(V),\kappa(A)\}+1}\}$.

\item[Step 3.] Set $\kappa=\max(\kappa(V),\kappa(A))$. If $\kappa>\kappa(V)$, simulate $V_{n}$ by sampling from $\bar{V}_{n}$ for
$n=\kappa(V)+2,...,\kappa+1$.

\item[Step 4.] Set $q=0$, $i=0$ and repeat the following procedure until $i=\kappa$:\\
{\it set $i=i+1$;  if $V_i > A_i$, set $q=q+1$ and $R_q=V_i-A_i$}.\\
Output $\{R_1,R_2,...R_q\}$ and $A_1$.
\end{itemize}

\subsection{Empirical Performance}

Let $Y=\{Y(t):t\geq0\}$ be a continuous time Markov process on the state space
$\Omega$ and $f$ is a real-valued function defined on $\Omega$. The ergodic theorem
guarantees in great generality (assuming a unique stationary distribution
$\pi\left(  \cdot\right)  $) that%
\[
\frac{1}{t}\int_{0}^{t}f(Y(s))ds\rightarrow\int_{\Omega}f(y)\pi
(dy)
\]
as $t\rightarrow\infty$ almost surely for every positive, measurable function
$f\left(  \cdot\right)  $. In the setting of the infinite server queue such a
stationary distribution exists if $EV_{n}<\infty$ and $EX_{n}<\infty$. The
most natural estimator for $E_{\pi}f(Y):=\int_{\Omega}f(y)\pi(dy)$ is therefore
\[
\Phi(t,Y(0)):=\frac{1}{t}\int_{0}^{t}f(Y(s))ds,
\]
where $Y(0)$ is the initial state. The estimator $\Phi\left(  t, Y(0) \right)  $ is generally biased unless
$Y\left(  0\right)  $ is sampled from the stationary distribution $\pi\left(\cdot\right)$ (\cite{AsmGlyn:2007} P. 97).
Our algorithm has the obvious advantage of removing the initial transient.

In what follows we conduct some simulation experiment to evaluate the
practical performance of our algorithm. The idea is to fix a reasonable
tolerance error, say 10\%, for a given performance measure. Then we want to
empirically find how large a burn-in period one would need in practice to
reduce the initial transient bias to about 10\%. In order to effectively
quantify the error we select a class of systems for which $\pi\left(
\cdot\right)  $ can be explicitly evaluated.

We consider an infinite server queue with Poisson arrivals and Lognormal
service times. As we are interested in the efficiency of our algorithm for
relatively large systems, we set the arrival rate $\lambda=100$ and the
service time $V_{n}\sim$ Lognormal$(-0.25,0.5)$ (i.e. $V_{n}$ has the same
distribution as $\exp\left(  -.25+.5\times N(0,1)\right)  $, where $N\left(
0,1\right)  $ denotes a standard Gaussian random variable).

Let $Y(t)=(Q(t,\cdot),E(t))\in\mathcal{D}[0,\infty)\times \mathbb{R}_{+}$, then $Y(t)$
is a Markovian measure valued descriptor of the infinite server queue (of
course in the Poisson arrival case one does not need to keep track of
$A\left(  \cdot\right)  $).

We first compare the performance of our algorithm to the burn-in period
defined as the period needed to reduce the initial transient as indicated
earlier. Let $f(Y(t))=Q(t,0)$, i.e. the number of people in the system at time
$t$. We measure the computation effort of the algorithm in terms of the number
of arrivals (we call this the number of steps) simulated. Given $\epsilon>0$
we let $n(\epsilon)$ denote the minimum number of steps required so that
$|E\Phi(A_{n(\epsilon)},(\phi,0))-E_{\pi}Q(0,0)|/E_{\pi} Q(0,0) \leq\epsilon$,
where $(\phi,0)$ denotes a system that starts empty with $E(0)=0$ (recall that $E(\cdot)$ is the age process associated with $N(\cdot)$, i.e. when $E(0)=x$,  $A_1$ is distributed as $X_n$ conditional on $X_n>x$).
Table~\ref{tab: first} shows the relation between $\epsilon$
and $n(\epsilon)$, obtained empirically based on the average of $10^{4}$
independent replications

\begin{table}[htb]
\centering
\caption{Bias of $\Phi(S_{n(\epsilon)})$ \label{tab: first}}
\begin{tabular}{c|c|c} \hline
$\epsilon$ & $n(\epsilon)$ & computer time (s)\\ \hline
$10.26\%$ & $6 \times 10^2$ & $0.0310$\\ \hline
$5.71\%$ & $1 \times 10^3$ & $0.0382$\\ \hline
$1.17\%$ & $5 \times 10^3$ & $0.1367$\\ \hline
\end{tabular}
\end{table}

Compared to the results in Table~\ref{tab: first}, our algorithm is unbiased. The average
number of steps involved is $n=592.6369$  based on the average of $10^{4}$ independent
replications  and the average computer time needed for a single
replication is $0.0249$ s.

In addition, in Table~\ref{tab: second} we compare the performance of the estimators  $\Phi(A_n, (\phi,0))$ and $\Phi(A_{n'}, (Q(0,\cdot),A_1))$,
where $Q(0,\cdot)$ and $A_1$ are sampled according to Algorithm 5.
 $n$ and $n'$ are calibrated so that the computation budget
is basically the same in both estimators. Under our procedure, $E\kappa$, the average number of arrivals required to terminate is approximately equal
to $600$. So for instance, the first row in Table 2 corresponds to $n=10^4$. This means that $n' \approx 9.4\times10^3=10^4-600$.
The true value of $E_{\pi}Q(0,0)$ is $88.2497$. The sample mean and sample standard deviation are calculated using the method of Batch means.
The result in Table~\ref{tab: second} shows that our mixed method performs better than the batch means with relatively small
computation budget, while with large budget, the two methods are about the same.

\begin{table}[htb]
\centering
\caption{Simulation result with different initial states. \label{tab: second}}
\begin{tabular}{c|c|c|c|c}\hline
& \multicolumn{2}{|c|}{$(\phi,0)$} & \multicolumn {2}{|c}{$(Q(0,\cdot), A_1)$} \\ \hline
$n$&  Sample Mean & Sample Std & Sample Mean & Sample Std \\ \hline
$1 \times 10^4$ & $86.1274$  & $1.0104$  & $88.1713$  & $0.6018$  \\ \hline
$5 \times 10^4$ & $89.0893$  & $0.4587$  & $88.2956$  & $0.3770$  \\ \hline
$1 \times 10^5$ & $88.5151$  & $0.3531$  & $88.1270$  & $0.2976$  \\ \hline
$5 \times 10^5$ & $88.3022$  & $0.1481$  & $88.3581$  & $0.1402$  \\ \hline
\end{tabular}
\end{table}

\section{Application to sensitivity analysis of the infinite server queue} \label{sec:sens}
In this section, we apply our algorithm to sensitivity analysis of the
infinite server queue. We consider a sequence of systems indexed by
$(\lambda,\nu),$ $\lambda>0,$ $\nu>0$. Given $\left(  \lambda,\nu\right)  $,
the interarrival times are multiplied by $1/\lambda$, obtaining $X_{n}%
/\lambda$ for all $n$, and the service times are multiplied by $1/\nu$, thus
we have $V_{n}/\nu$ for all $n$. We assume that $EV_{n}<\infty$ and
$EX_{n}<\infty$. We will use the notation $Q_{\lambda,\nu}\left(
\cdot\right)  $ to denote the infinite server queue descriptor for the
$\left(  \lambda,\nu\right)  $-system. Our strategy rests on the application
of Infinitesimal Perturbation Analysis (IPA), see for instance \cite{Glass:2003} P. 386.
We assume here that the interarrival times have a continuous
distribution.

We illustrate the methodology by computing the sensitivity of
the steady-state average remaining service time, which we denote by $E_{\pi
}\bar{R}(\lambda,\nu)$; namely,%
\[
E_{\pi}\bar{R}(\lambda,\nu)=E_{\pi}\frac{1}{Q_{\lambda,\nu}\left(  0,0\right)
}\int_{0}^{\infty}yQ_{\lambda,\nu}\left(  0,dy\right)  .
\]
We also consider
\[
E_{\pi}R^{\infty}(\lambda,\nu)=E_{\pi}(\inf\{y\geq0:Q_{\lambda,\nu}\left(
0,y\right)  =0\}),
\]
in words, the steady-state maximum remaining service time. In order to apply
IPA we need to define a few quantities.

First, let us define $\bar{\Xi}(\lambda,\nu)$ to be the average elapsed
service time of the customers that are present at time zero (given the
construction of the stationary process $\{  Q_{\lambda,\nu}\left(
t,\cdot\right)  :t\in\left(  -\infty,\infty\right)  \}  $, see Figure
2). That is,
\[
\bar{\Xi}(\lambda,\nu)=\frac{1}{Q_{\lambda,\nu}\left(  0,0\right)  }%
\sum_{n=-1}^{-\infty}\frac{\left\vert A_{n}\right\vert }{\lambda}I\left(
\frac{\left\vert A_{n}\right\vert }{\lambda}<\frac{V_{n}}{\nu}\right)
\]
Likewise, define $\bar{V}(\lambda,\nu)$ as the average of the total service
requirement of the customers that are present at time zero, namely
\[
\bar{V}(\lambda,\nu)=\frac{1}{Q_{\lambda,\nu}\left(  0,0\right)  }\sum
_{n=-1}^{-\infty}\frac{V_{n}}{v}I\left(  \frac{\left\vert A_{n}\right\vert}
{\lambda}<\frac{V_{n}}{\nu}\right)  .
\]
Next, we define $\Xi^{(\infty)}\left(  \lambda,\nu\right) $ as the elapsed service time of the customer with
the maximum remaining service time at time zero and $V^{(\infty)}(\lambda,\nu)$ as his total service time requirement. Specifically, if we let
$m=\arg\max\{n: V_n/\nu-|A_n|/\lambda \}$
then
$$\Xi^{(\infty)}\left(  \lambda,\nu\right)=\frac{|A_m|}{\lambda} \mbox{ and } V^{(\infty)}(\lambda,\nu)=\frac{V_m}{\nu}$$

We then obtain the following representation for the derivatives of $E_{\pi
}\bar{R}(\lambda,\nu)$ and $E_{\pi} R^{\infty}(\lambda,\nu)$ with respect to $\lambda$ and $\nu$.

\begin{lemma}
We have that\\
i)
\[
\frac{\partial}{\partial\lambda}E_{\pi}\bar{R}(\lambda,\nu)=\frac{1}%
{\lambda}E_{\pi}\bar{\Xi}(\lambda,\nu)
\mbox{ and }
\frac{\partial}{\partial\nu}E_{\pi}\bar{R}(\lambda,\nu)=-\frac{1}{\nu%
}E_{\pi}\bar{V}(\lambda,\nu).
\]
ii)
\[
\frac{\partial}{\partial\lambda} E_{\pi}R^{\infty}(\lambda,\nu)=\frac
{1}{\lambda}E_{\pi}\Xi^{(\infty)}(\lambda,\nu)
\mbox{ and  }
\frac{\partial}{\partial\nu} E_{\pi}R^{\infty}(\lambda,\nu)=-\frac{1}{\nu%
}E_{\pi}V^{(\infty)} (\lambda,\nu)
\]
\end{lemma}

\begin{proof}
We only give a proof of part i) here as the proof of part ii) is entirely analogous.\\
Let $R_n$ denote the remaining service time of the $n$th customer at time zero and $V_n$ as his total service time requirement, then $R_n \leq V_n$. Thus if $EV_n<\infty$, we have
$$E_{\pi}\bar R(\lambda,\nu)<\infty$$
for any $\lambda>0$,$\nu>0$.\\
For a fixed sample path $\omega$ constructed backward in time, let $R_{n}(\lambda,\nu, \omega)$, $n < 0$, denote the remaining service time of customer $n$ (counting backward in time) at time $0$ in system $(\lambda,\nu)$. Then $R_{n}(\lambda,\nu,\omega)=(V_{n}(\omega)/\nu-|A_{n}(\omega)|/\lambda)^+$ and
$$\lim_{h \rightarrow 0} \frac{R_n(\lambda+h,\nu,\omega)-R_n(\lambda,\nu,\omega)}{h}=\frac{|A_n(\omega)|}{\lambda^2} 1\{\frac{V_n(\omega)}{\nu} \geq \frac{|A_n(\omega)|}{\lambda}\}$$
$$\lim_{h \rightarrow 0} \frac{R_n(\lambda,\nu+h,\omega)-R_n(\lambda,\nu,\omega)}{h}=-\frac{V_n(\omega)}{\nu^2} 1\{\frac{V_n(\omega)}{\nu}\geq \frac{|A_n(\omega)|}{\lambda}\}$$
Thus the derivative $\frac{\partial}{\partial\lambda} \bar R(\lambda,\nu)$ and  $\frac{\partial}{\partial\nu} \bar R(\lambda,\nu)$ exists.\\
Let $\Xi_n$ denote the elapsed service time of the $n$th customer at time zeros and define $\Xi_n=V_n$ if he is no longer in the system at time zero, then $\Xi _n \leq V_n$. Therefore $E_\pi \frac{\partial}{\partial\lambda} \bar R(\lambda,\nu)<\infty$ and $E_\pi \frac{\partial}{\partial\nu} \bar R(\lambda,\nu)<\infty$.\\
As $|(\bar R_n(\lambda+h,\nu)-\bar R_n(\lambda,\nu))/h| \leq \max_{\kappa_{\lambda+h,\nu} < n <0} V_n/\lambda^2 $ and $|(\bar R_n(\lambda,\nu+h)-\bar R_n(\lambda,\nu))/h| \leq \max_{\kappa_{\lambda,\nu+h} < n <0} V_n/\nu^2$, by Lebesgue Dominated Convergence Theorem, we have
$$\frac{\partial}{\partial\lambda} E_\pi \bar R(\lambda,\nu)=E_\pi \frac{\partial}{\partial\lambda} \bar R(\lambda,\nu)
\mbox{ and }
\frac{\partial}{\partial\nu} E_\pi \bar R(\lambda,\nu)=E_\pi \frac{\partial}{\partial\nu} \bar R(\lambda,\nu)$$
As the interarrival times have a continuous distribution,  $P(V_n/\nu=|A_n|/\lambda)=0$ for $n < 0$.\\
Combining the change of limit and the sample path analysis we have
$$\frac{\partial}{\partial\lambda} E_\pi \bar R(\lambda,\nu)=\frac{1}{\lambda}E_\pi \bar \Xi(\lambda,\nu)
\mbox{ and  }
\frac{\partial}{\partial\nu} E_\pi \bar R(\lambda,\nu)=-\frac{1}{\nu}E_\pi \bar V(\lambda,\nu)$$
\end{proof}

Table 3 shows the simulated results of an infinite server queue with base
(i.e. $\lambda=1$) interarrival times distributed as Gamma$(2,2)$ and base
(i.e. $\nu=1$) service times distributed as Lognormal$(-0.25,0.5)$.

\begin{table}[htb]
\centering
\caption{Simulation result from exact sampling. }
\begin{tabular}{c|c|c|c|c}\hline
$(\lambda,\nu)$& $\frac{\partial}{\partial\lambda} E_\pi \bar R(\lambda,\nu)$ & $\frac{\partial}{\partial\nu} E_\pi \bar R(\lambda,\nu)$ & $\frac{\partial}{\partial\lambda} E_\pi R^{\infty}(\lambda,\nu)$ & $\frac{\partial}{\partial\nu} E_\pi R^{\infty}(\lambda,\nu)$ \\ \hline
$(80,1)$ & $7.0741 \times 10^{-3}$  & $-1.1320$  & $6.1022 \times 10^{-3}$  & $-2.8389$  \\ \hline
$(100,1)$ & $5.6470 \times10^{-3}$  & $-1.1316$  & $4.9379 \times 10^{-3}$  & $-2.9495$  \\ \hline
$(120,1)$ & $4.7236 \times 10^{-3}$  & $-1.1337$  & $4.2337 \times 10^{-3}$  & $-3.0684$  \\ \hline
\end{tabular}
\end{table}

\section*{ACKNOWLEDGMENTS}
Support from the NSF foundation through the grants CMMI-0846816 and CMMI-1069064 is gratefully acknowledged.

\bibliographystyle{plain}

\bibliography{exactBib}

\end{document}